\documentclass[12pt]{amsart}

\usepackage[latin1]{inputenc}
\usepackage[english]{babel}
\usepackage{amsthm}
\usepackage{amsmath}
\usepackage{amsxtra}
\usepackage{mathrsfs}
\usepackage{amssymb}
\usepackage[dvips]{graphicx}
\usepackage{graphicx}

\input xy
\xyoption{all}

\newcommand{\lev}{{\rm Lev}}
\newcommand{\tail}{{\rm Tail}}

\newcommand{\TP}{{\rm TP }}
\newcommand{\ITP}{{\rm ITP }}
\newcommand{\dom}{{\rm dom}}

\newcommand{\ZFC}{{\rm ZFC }}

\newcommand{\Add}{{\rm Add}}

\newcommand{\force}{\Vdash}
\newcommand{\spazio}{\textrm{ }}
\newcommand{\restr}{\upharpoonright}

\newtheorem{theorem}{Theorem}[section]
\newtheorem{lemma}[theorem]{Lemma}
\newtheorem{proposition}[theorem]{Proposition}
\newtheorem{coroll}[theorem]{Corollary}
\newtheorem{remark}[theorem]{Remark}
\newtheorem{definition}[theorem]{Definition}
\newtheorem{notation}[theorem]{Notation}
\newtheorem{claim}[theorem]{Claim}
\begin{document}
\title{Strong Tree Properties For Small Cardinals}


\author[Laura Fontanella ]{Laura Fontanella}
\address{Equipe de Logique Math\'ematique,
Universit\'e Paris Diderot Paris 7, UFR de math\'ematiques case 7012, 
site Chevaleret, 75205 Paris Cedex 13, France}
\email{fontanella@logique.jussieu.fr}


\subjclass[2010]{03E55 }

\keywords{tree property, large cardinals, forcing.}

\date{9 february 2012}



\maketitle

\begin{abstract}{}
An inaccessible cardinal $\kappa$ is supercompact when $(\kappa, \lambda)$-ITP holds for all $\lambda\geq \kappa.$ We prove that if there is a model of $\ZFC$ with infinitely many supercompact cardinals, then there is a model of \ZFC where for every $n\geq 2$ and $\mu\geq \aleph_n,$ we have $(\aleph_n, \mu)$-ITP. 
   \end{abstract}

\




\section{Introduction}

One of the most intriguing research axes in contemporary set theory is the investigation into those properties which are typically associated with large cardinals, though they can be satisfied by small cardinals as well. The tree property is a principle of that sort. Given a regular cardinal $\kappa,$ we say that $\kappa$ satisfies the \emph{tree property} when every $\kappa$-tree has a cofinal branch. The result presented in the present paper concerns the so-called \emph{strong tree property} and \emph{super tree property}, which are two combinatorial principles that generalize the usual tree property. The definition of those properties will be presented in \S \ref{sstp}, for now let us just discuss some general facts about their connection with large cardinals. We know that an inaccessible cardinal is weakly compact if, and only if, it satisfies the tree property. The strong and the super tree properties provide a similar characterization of strongly compact and supercompact cardinals, indeed an inaccessible cardinal is strongly compact if, and only if, it satisfies the strong tree property, while it is supercompact if, and only if, it satisfies the super tree property (the former result follows from a theorem by Jech \cite{Jech}, the latter is due to Magidor \cite{Magidor}). In other words, when a cardinal satisfies one  of the previous properties, it ``behaves like a large cardinal".\\ 


While the previous characterizations date back to the early $1970$s, a systematic study of the strong and the super tree properties has only recently been undertaken by Weiss
(see \cite{WeissPhd} and \cite{Weiss}). He proved in \cite{Weiss} that for every $n\geq 2,$ one can define a model of the super tree property for $\aleph_n,$ starting from a model with a supercompact cardinal. It is natural to ask whether all small cardinals (that is cardinals of the form $\aleph_n$ with $n\geq 2$) can \emph{simultaneously} have the strong or the super tree properties. Fontanella \cite{Fontanella} proved that a forcing construction due to Abraham \cite{Abraham} generalizes to show that the super tree property can hold for two successive cardinals. Cummings and Foreman \cite{CummingsForeman} proved that if there is a model of set theory with infinitely many supercompact cardinals, then one can obtain a model in which every $\aleph_n$ with $n\geq 2$ satisfies the tree property. In the present paper, we  prove that in the Cummings and Foreman's model even the \emph{super} tree property holds at every $\aleph_n$ with $n\geq 2.$ The same result has been proved independently by Unger \cite{Unger}.\\ 

The paper is organized as follows. In \S \ref{sstp} we introduce the strong and the super tree properties. \S \ref{BPP} is devoted to the proof of two preservation theorems. In \S \ref{mainforcing} we define Cummings and Foreman's model. In \S \ref{subsec}, \S \ref{the term forcing} and \S \ref{morepreserv}, we expand that model and we analyze some properties of the new generic extension. Finally, we prove in \S \ref{theorem} that in Cummings and Foreman's model every cardinal $\aleph_n$ (with $n\geq 2$) has the super tree property.

\section{Preliminaries and Notation}

Given a forcing $\mathbb{P}$ and conditions $p,q\in \mathbb{P},$ we use $p\leq q$ in the sense that $p$ is stronger than $q;$ we write $p\vert \vert q$ when $p$ and $q$ are two compatible conditions (i.e. there is a condition $r\in \mathbb{P}$ such that $r\leq p$ and $r\leq q$). A poset $\mathbb{P}$ is \emph{separative} if whenever 
$q\not\leq p,$ then some extension of $q$ in $\mathbb{P}$ is incompatible with $p.$ Every partial order can be turned into a separative poset. Indeed, one can define $p\prec q$ iff all extensions of $p$ are compatible with $q,$ then the resulting equivalence relation, given by $p\sim q$ iff $p\prec q$ and $q\prec p,$ provides a separative poset; we denote by $[p]$ the equivalence class of $p.$\\ 

A forcing $\mathbb{P}$ is \emph{$\kappa$-closed} if, and only if, every descending sequence of conditions of $\mathbb{P}$ of size less than $\kappa$ has a lower bound; 
$\mathbb{P}$ is \emph{$\kappa$-directed closed} if, and only if, for every set of less than $\kappa$ pairwise compatible conditions of $\mathbb{P}$ has a lower bound. We say that 
$\mathbb{P}$ is $<\kappa$-distributive if, and only if, no sequence of ordinals of length less than $\kappa$ is added by $\mathbb{P}.$ $\mathbb{P}$ is $\kappa$-c.c. when every 
antichain of $\mathbb{P}$ has size less than $\kappa;$ $\mathbb{P}$ is $\kappa$-Knaster if, and only if, for all sequence of conditions $\langle p_{\alpha};\ \alpha<\kappa \rangle,$ there is $X\subseteq \kappa$ cofinal such that the conditions of the sequence $\langle p_{\alpha};\ \alpha\in X \rangle$ are pairwise compatible.\\  

Given two forcings $\mathbb{P}$ and $\mathbb{Q},$ we will write 
$\mathbb{P}\equiv \mathbb{Q}$ when $\mathbb{P}$ and $\mathbb{Q}$ are equivalent, namely: 
\begin{enumerate}
\item for every filter $G_{\mathbb{P}}\subseteq \mathbb{P}$ which is generic over $V,$ there exists a filter $G_{\mathbb{Q}}\subseteq \mathbb{Q}$ which is generic over $V,$ and $V[G_{\mathbb{P}}]= V[G_{\mathbb{Q}}] ;$
\item for every filter $G_{\mathbb{Q}}\subseteq \mathbb{Q}$ which is generic over $V,$ there exists a filter $G_{\mathbb{P}}\subseteq \mathbb{P}$ which is generic over $V,$ and $V[G_{\mathbb{P}}]= V[G_{\mathbb{Q}}].$
\end{enumerate}

If $\mathbb{P}$ is any forcing and $\dot{\mathbb{Q}}$ is a $\mathbb{P}$-name for a forcing, then we denote by $\mathbb{P}\ast \dot{\mathbb{Q}}$ the poset 
$\{(p,q);\spazio p\in \mathbb{P}, q\in V^{\mathbb{P}}\textrm{ and }p\force q\in \dot{\mathbb{Q}}\},$
where for every $(p,q), (p',q')\in \mathbb{P}\ast \dot{\mathbb{Q}},$ $(p,q)\leq (p',q')$ if, and only if, $p\leq p'$ and $p\force q\leq q'.$\\

 If $\mathbb{P}$ and $\mathbb{Q}$ are two posets, a \emph{projection} $\pi: \mathbb{Q}\to \mathbb{P}$ is a function such that:
\begin{enumerate}
\item for all $q,q'\in \mathbb{Q},$ if $q\leq q',$ then $\pi(q)\leq \pi(q');$
\item $\pi(1_\mathbb{Q})= 1_{\mathbb{P}};$
\item for all $q\in \mathbb{Q},$ if $p\leq \pi(q),$ then there is $q'\leq q$ such that $\pi(q')\leq p.$ 
\end{enumerate}
We say that \emph{$\mathbb{P}$ is a projection of $\mathbb{Q}$} when there is a projection $\pi: \mathbb{Q}\to \mathbb{P}.$

If $\pi: \mathbb{Q}\to \mathbb{P}$ is a projection and $G_\mathbb{P}\subseteq \mathbb{P}$ is a generic filter over $V,$ define
$$\mathbb{Q}/G_{\mathbb{P}}:=\{ q\in \mathbb{Q};\ \pi(q)\in G_\mathbb{P}\},$$
$\mathbb{Q}/G_\mathbb{P}$ is ordered as a subposet of $\mathbb{Q}.$ The following hold: 
\begin{enumerate}
\item If $G_{\mathbb{Q}}\subseteq \mathbb{Q}$ is a generic filter over  $V$ and $H:= \{p\in \mathbb{P};\ \exists q\in G_{\mathbb{Q}}(\pi(q)\leq p)  \},$ then $H$ is $\mathbb{P}$-generic over $V;$
\item if $G_{\mathbb{P}}\subseteq \mathbb{P}$ is a generic filter over $V,$ and if $G\subseteq \mathbb{Q}/G_\mathbb{P}$ is a generic filter over $V[\mathbb{G_{\mathbb{P}}}],$ then $G$ is $\mathbb{Q}$-generic over $V,$ and $\pi[G]$ generates $G_{\mathbb{P}};$
\item if $G_{\mathbb{Q}}\subseteq \mathbb{Q}$ is a generic filter, and $H:= \{p\in \mathbb{P};\ \exists q\in G_{\mathbb{Q}}(\pi(q)\leq p) \},$ then $G_{\mathbb{Q}}$ is 
$\mathbb{Q}/G_\mathbb{P}$-generic over $V[H].$ That is, we can factor forcing with $\mathbb{Q}$ as forcing with $\mathbb{P}$ followed by forcing with 
$\mathbb{Q}/G_\mathbb{P}$ over $V[G_{\mathbb{P}}].$
\end{enumerate}

Some of our projections $\pi: \mathbb{Q}\to \mathbb{P}$ will also have the following property: for all $p\leq \pi(q),$ there is $q'\leq q$ such that 
\begin{enumerate}
\item $\pi(q')= p,$  
\item for every $q^*\leq q,$ if $\pi(q^*)\leq p,$ then $q^*\leq q'.$
\end{enumerate}

\

Let $\kappa$ be a regular cardinal and $\lambda$ an ordinal, we denote by $\Add(\kappa, \lambda)$ the poset of all partial functions $f:\lambda\to 2$ of size less than $\kappa,$ ordered by reverse inclusion. We use $\Add(\kappa)$ to denote $\Add(\kappa, \kappa).$\\ 


If $V\subseteq W$ are two models of set theory with the same ordinals and $\eta$ is a cardinal in $W,$ we say that $(V,W)$ has the $\eta$-covering property if, and only if, every set $X\subseteq V$ in $W$ of cardinality less than $\eta$ in $W,$ is contained in a set $Y\in V$ of cardinality less than $\eta$ in $V.$\\ 



Assume that $\mathbb{P}$ is a forcing notion in a model $V,$ we will use $V[\mathbb{P}]$ to denote a generic extension by some unspecified $\mathbb{P}$-generic filter.

\begin{lemma} (Easton's Lemma) Let $\kappa$ be regular. If $\mathbb{P}$ has the $\kappa$-chain condition and $\mathbb{Q}$ is $\kappa$-closed, then 
\begin{enumerate}
\item $\force_{\mathbb{Q}} \mathbb{P}\textrm{ has the $\kappa$-chain condition};$ 
\item $\force_{\mathbb{P}} \mathbb{Q}\textrm{ is a $<\kappa$-distributive};$
\item If $G$ is $\mathbb{P}$-generic over $V$ and $H$ is $\mathbb{Q}$-generic over $V,$ then $G$ and $H$ are mutually generic;
\item If $G$ is $\mathbb{P}$-generic over $V$ and $H$ is $\mathbb{Q}$-generic over $V,$ then $(V, V[G][H])$ has the $\kappa$-covering property;
\item If $\mathbb{R}$ is $\kappa$-closed, then $\force_{\mathbb{P}\times \mathbb{Q}} \textrm{ $\mathbb{R}$ is $<\kappa$-distributive}.$
\end{enumerate}  
\end{lemma}

For a proof of that lemma see \cite[Lemma 2.11]{CummingsForeman}.\\  

Let $\eta$ be a regular cardinal, $\theta> \eta$ be large enough and $M\prec H_\theta$ of size $\eta$. We say that $M$ is \emph{internally approachable of length $\eta$} if it can be written as the union of an increasing continuous chain  $\langle M_{\xi}: \xi < \eta\rangle $ of elementary submodels of $H(\theta)$ of size less than $\eta,$ such that $\langle M_{\xi}: \xi < \eta' \rangle  \in M_{\eta' +1}$, for every ordinal $\eta' < \eta.$\\





We will assume familiarity with the theory of large cardinals and elementary embeddings, as developed for example in \cite{Kanamori}. 

\begin{lemma} (Laver) \cite{Laver} If $\kappa$ is a supercompact cardinal, then there exists $L: \kappa \to V_{\kappa}$ such that: for all $\lambda,$ for all $x\in H_{\lambda^+},$ there is an elementary embedding $j: V\to M$ with critical point $\kappa$
such that $j(\kappa)>\lambda,$ ${}^\lambda M\subseteq M$ and $j(L)(\kappa)= x.$
\end{lemma}

\begin{lemma} (Silver) Let $j: M\to N$ be an elementary embedding between inner models of {\rm ZFC}. Let $\mathbb{P}\in M$ be a forcing and suppose that $G$ is $\mathbb{P}$-generic over $M,$ $H$ is $j(\mathbb{P})$-generic over $N,$ and $j[G]\subseteq H.$ Then, there is a unique $j^*: M[G]\to N[H]$ such that $j^*\restr M= j$ and $j^*(G)= H.$ 
\end{lemma}

\begin{proof} If $j[G]\subseteq H,$ then the map $j^*(\dot{x}^{G})= j(\dot{x})^{H}$ is well defined and satisfies the required properties. \end{proof}

\section{The Strong and the Super Tree Properties}\label{sstp}

We recall the definition of the tree property, for a regular cardinal $\kappa.$

\begin{definition} Let $\kappa$ be a regular cardinal, 
\begin{enumerate}
\item a $\kappa$-tree is a tree of height $\kappa$ with levels of size less than $\kappa;$
\item we say that $\kappa$ has the \emph{tree property} if, and only if, every $\kappa$-tree has a cofinal branch (i.e. a branch of size $\kappa$).  
\end{enumerate}
\end{definition}

The strong and the super tree property concern special objects that generalize the notion of $\kappa$-tree, for a regular cardinal $\kappa.$ 

\begin{definition}\label{main definition} Given $\kappa\geq \omega_2$ a regular cardinal and $\lambda\geq \kappa,$ a \emph{$(\kappa, \lambda)$-tree} is a set $F$ satisfying the following properties:  
\begin{enumerate}
\item for every $f\in F,$ $f: X\to 2,$ for some $X\in [\lambda]^{<\kappa}$
\item for all $f\in F,$ if $X\subseteq \dom(f),$ then $f\restr X\in F;$
\item the set $\lev_X(F):= \{f\in F;\spazio \dom(f)=X \}$ is non empty, for all $X\in [\lambda]^{<\kappa};$
\item $\vert \lev_X(F) \vert<\kappa ,$ for all $X\in [\lambda]^{<\kappa}.$
\end{enumerate}
\end{definition}

When there is no ambiguity, we will simply write $\lev_X$ instead of $\lev_X(F).$ The main difference between $\kappa$-trees and $(\kappa, \lambda)$-trees is the fact that, 
in the former, levels are indexed by ordinals, while in the latter, levels are indexed by \emph{sets of ordinals}. Therefore, the ordering between the levels of a $(\kappa, \lambda)$-tree is not total.  

\begin{definition}\label{branches} Given $\kappa\geq \omega_2$ a regular cardinal, $\lambda\geq \kappa,$ and a $(\kappa, \lambda)$-tree $F,$
\begin{enumerate}
\item  a \emph{cofinal branch} for $F$ is a function $b: \lambda \to 2$ such that $b\restr X\in \lev_X(F),$ for all $X\in[\lambda]^{<\kappa};$
\item an \emph{$F$-level sequence} is a function $D: [\lambda]^{<\kappa}\to F$ such that for every $X\in [\lambda]^{<\kappa},$ $D(X)\in \lev_X(F);$
\item given an $F$-level sequence $D,$ an \emph{ineffable branch} for $D$ is a cofinal branch $b: \lambda \to 2$ such that
$\{X\in [\lambda]^{<\kappa};\spazio b\restr X= D(X) \}$ is stationary. 
\end{enumerate}
\end{definition}

\begin{definition} Given $\kappa\geq \omega_2$ a regular cardinal and $\lambda\geq \kappa,$ 
\begin{enumerate}
\item $(\kappa, \lambda)$-\TP holds if every $(\kappa, \lambda)$-tree has a cofinal branch;
\item $(\kappa, \lambda)$-\ITP holds if for every $(\kappa, \lambda)$-tree $F$ and for every $F$-level sequence $D,$ there is an an ineffable branch for $D;$
\item we say that $\kappa$ satisfies the \emph{strong tree property} if $(\kappa, \mu)$-\TP holds, for all $\mu\geq \kappa;$
\item we say that $\kappa$ satisfies the \emph{super tree property} if 
$(\kappa,\mu)$-\ITP holds, for all $\mu\geq \kappa;$
\end{enumerate}
\end{definition}


\section{The Preservation Theorems}\label{BPP}

It will be important, in what follows, that certain forcings cannot add ineffable branches. The following proposition is due to Silver (see \cite[chap. VIII, Lemma $3.4$]{Kunen} or \cite[Proposition 2.1.12]{WeissPhd}), we include the proof for completeness.

\begin{theorem}\label{closed}(First Preservation Theorem) Let $\theta$ be a regular cardinal and $\mu\geq \theta$ be any ordinal. Assume that $F$ is a $(\theta, \mu)$-tree and $\mathbb{Q}$ is an $\eta^+$-closed forcing with $\eta< \theta\leq 2^{\eta}.$ For every filter $G_{\mathbb{Q}}\subseteq \mathbb{Q}$ generic over $V,$ every cofinal branch for $F$ in $V[G_{\mathbb{Q}}]$ is already in $V.$
\end{theorem}

\begin{proof} We can assume, without loss of generality, that $\eta$ is minimal such that $2^{\eta}\geq \theta.$ Assume towards a contradiction that $\mathbb{Q}$ adds a cofinal branch to $F,$ let $\dot{b}$ be a $\mathbb{Q}$-name for such a function.  
For all $\alpha\leq \eta$ and all $s\in {}^{\alpha}2,$ we are going to define by induction three objects $a_{\alpha}\in [\mu]^{<\theta},$ $f_s\in \lev_{a_{\alpha}}$ and $p_s\in \mathbb{Q}$ such that:

\begin{enumerate}
\item $p_s\force \dot{b}\restr a_{\alpha}= f_s;$ 
\item $f_{s\smallfrown 0}(\beta)\neq f_{s\smallfrown 1}(\beta),$ for some $\beta<\mu;$
\item if $s\subseteq t,$ then $p_t\leq p_s;$
\item if $\alpha< \beta,$ then $a_{\alpha}\subset a_{\beta}.$
\end{enumerate}

Let $\alpha<\eta,$ assume that $a_{\alpha}, f_s$ and $p_s$ have been defined for all $s\in {}^{\alpha}2.$ We define $a_{\alpha+1},$ $f_s,$ and $p_s,$ for all $s\in {}^{\alpha+1}2.$ Let $t$ be in ${}^{\alpha}2,$ we can find an ordinal $\beta_t\in \mu$ and two conditions $p_{t\smallfrown 0}, p_{t\smallfrown 1}\leq p_t$ such that $p_{t\smallfrown 0}\force \dot{b}(\beta_t)= 0$ and $p_{t\smallfrown 1}\force \dot{b}(\beta_t)= 1.$ (otherwise, $\dot{b}$ would be a name for a cofinal branch which is already in $V$).  Let 
$a_{\alpha+1}:=a_{\alpha}\cup \{\beta_t;\ t\in {}^{\alpha}2\},$ then $\vert a_{\alpha+1}\vert <\theta,$ because $2^{\alpha}<\theta$. We just defined, for every $s\in {}^{\alpha+1}2,$ a condition $p_s.$ Now, by strengthening $p_s$ if necessary, we can find $f_s\in \lev_{a_{\alpha+1}}$ such that 
$$p_s\force \dot{b}\restr a_{\alpha+1}= f_s.$$
Finally, $f_{t\smallfrown 0}(\beta_t)\neq f_{t\smallfrown 1}(\beta_t),$ for all $t\in {}^{\alpha}2:$ because $p_{t\smallfrown 0}\force f_{t\smallfrown 0}(\beta_t)= \dot{b}(\beta_t)=0,$ while $p_{t\smallfrown 1}\force f_{t\smallfrown 1}(\beta_t)= \dot{b}(\beta_t)=1.$\\

If $\alpha$ is a limit ordinal $\leq \eta,$ let $t$ be any function in ${}^{\alpha}2.$ Since $\mathbb{Q}$ is $\eta^+$-closed, there is a condition $p_t$ such that $p_t\leq p_{t\restr {\beta}},$ for all $\beta<\alpha.$ Define $a_{\alpha}:= \underset{\beta<\alpha}{\bigcup} a_{\beta}.$ By strengthening $p_t$ if necessary, we can find $f_t\in \lev_{a_{\alpha}}$ such that 
$p_t\force \dot{b}\restr a_{\alpha}=f_t.$ That completes the construction. \\

We show that $\vert \lev_{a_{\eta}}\vert\geq {}^{\eta}2\geq \theta,$ thus a contradiction is obtained. Let $s\neq t$ be two functions in ${}^{\eta}2,$ we are going to prove that 
$f_s\neq f_t.$ Let $\alpha$ be the minimum ordinal less than $\eta$ such that $s(\alpha)\neq t(\alpha),$ without loss of generality $r\smallfrown 0\sqsubset s$ and $r\smallfrown 1\sqsubset t,$ for some $r\in {}^{\alpha}2.$ By construction, 
$$p_s\leq p_{r\smallfrown 0}\force \dot{b}\restr a_{\alpha+1}= f_{r\smallfrown 0}\textrm{ and }p_t\leq p_{r\smallfrown 1}\force \dot{b}\restr a_{\alpha+1}= f_{r\smallfrown 1},$$ where 
$f_{r\smallfrown 0}(\beta)\neq f_{r\smallfrown 1}(\beta),$ for some $\beta.$ Moreover, $p_s\force \dot{b}\restr a_{\eta}=f_s$ and $p_t\force \dot{b}\restr a_{\eta}= f_t,$ hence 
$f_s\restr a_{\alpha+1}(\beta)= f_{r\smallfrown 0}(\beta)\neq f_{r\smallfrown 1}(\beta)= f_t\restr a_{\alpha+1}(\beta),$ thus $f_s\neq f_t.$ That completes the proof. \end{proof}

The following theorem is rather ad hoc. It will be used several times in the final theorem.  

\begin{theorem}\label{mafalda}(Second Preservation Theorem) Let $V\subseteq W$ be two models of set theory with the same ordinals and let $\mathbb{P}\in V$ be a forcing notion and $\kappa$ a cardinal in $V$ such that: 
\begin{enumerate}
\item $\mathbb{P}\subseteq \Add(\aleph_n, \tau)^V,$ for some $\tau>\aleph_n,$\\
 and for every $p\in \mathbb{P},$ if $X\subseteq \dom(p),$ then $p\restr X\in \mathbb{P};$
\item $\aleph_{m}^V=\aleph_m^W,$ for every $m\leq n,$ and $W\models \vert \kappa \vert=\aleph_{n+1};$
\item for every set $X\subseteq V$ in $W$ of size $<\aleph_{n+1}$ in $W,$ there is $Y\in V$ of size $<\kappa$ in $V,$ such that $X\subseteq Y;$
\item in $V,$ we have $\gamma^{<\aleph_n}<\kappa,$ for every cardinal $\gamma<\kappa.$ 
\end{enumerate}
Let $F\in W$ be a $(\aleph_{n+1},\mu )$-tree with $\mu\geq \aleph_{n+1},$ then for every filter $G_{\mathbb{P}}\subseteq \mathbb{P}$ generic over $W,$ every cofinal branch for $F$ in $W[G_{\mathbb{P}}]$ is already in $W.$
\end{theorem}

\begin{proof} Work in $W.$ Let $\dot{b}\in W^{\mathbb{P}}$ and let $p\in \mathbb{P}$ such that 
$$p\force \dot{b}\textrm{ is a cofinal branch for $F.$}$$ 

We are going to find a condition $q\in \mathbb{P}$ such that $q\vert \vert p$ and for some $b\in W,$ we have $q\force \dot{b}= b.$
Let $\chi$ be large enough, for all $X\prec H_{\chi}$ of size $\aleph_n,$ we fix a condition $p_X\leq p$ and a function $f_X\in Lev_{X\cap\mu}$ such that 
$$p_X\force \dot{b}\restr X= f_X.$$
 Let $S$ be the set of all the structures $X\prec H_{\chi},$ such that $X$ is internally approachable of length $\aleph_n.$ 
Since every condition of $\mathbb{P}$ has size less than $\aleph_n,$ there is, for all $X\in S,$ a set $M_X\in X$ of size less than $\aleph_n$ such that 
$$p_X\restr X\subseteq M_X.$$
  By the Pressing Down Lemma, there exists $M^*$ and a stationary set $E^*\subseteq S$ such that $M^*=M_X,$ for all $X\in E^*.$ The set $M^*$ has size less than $\aleph_n$ in $W,$ hence $A:= (\underset{X\in E^*}{\bigcup} p_X)\restr M^*$ has size less than $\aleph_n$ in $W.$ By the assumption, $A$ is covered by some $N\in V$ of size 
$\gamma<\kappa$ in $V.$ In $V,$ we have 
$\vert [N]^{<\aleph_n}\vert\leq  \gamma^{<\aleph_n}< \kappa.$ It follows that in $W$ there are less than $\aleph_{n+1}$ possible values for 
$p_X\restr M^*.$ Therefore, we can find in $W$ a cofinal $E\subseteq E^*$ and a condition $q\in \mathbb{P},$ such that 
$p_X\restr X=q,$ for all $X\in E.$

\begin{claim} $f_X\restr Y= f_Y\restr X,$ for all $X,Y\in E.$
\end{claim}

\begin{proof} Let $X,Y\in E,$ there is $Z\in E$ with $X,Y,\dom(p_X), \dom(p_Y)\subseteq Z.$ Then, we have $p_X\cap p_Z= p_X\cap (p_Z\restr Z)= p_X\cap q= q,$ thus $p_X\vert \vert p_Z$ and similarly $p_Y\vert\vert p_Z.$ Let $r\leq p_X, p_Z$ and $s\leq p_Y, p_Z,$ then 
$r\force f_Z\restr X= \dot{b}\restr X= f_X$ and $s\force f_Z\restr Y= \dot{b}\restr Y= f_Y.$ It follows that $f_X\restr Y=f_Z\restr (X\cap Y)= f_Y\restr X.$ \end{proof} 

Let $b$ be $\underset{X\in E}{\bigcup} f_X.$ The previous claim implies that $b$ is a function and $$b\restr X= f_X,\textrm{ for all }X\in E.$$ 

\begin{claim} $q\force \dot{b}= b.$
\end{claim}

\begin{proof} We show that for every $X\in E,$ the set $B_X:= \{ s\in \mathbb{P};\ s\force \dot{b}\restr X= b\restr X \}$ is dense below $q.$ Let $r\leq q,$ there is $Y\in E$ such that 
$\dom(r), X\subseteq Y.$ 
It follows that $p_Y\cap r= p_Y\restr Y\cap r= q\cap r= q,$ thus $p_Y\vert \vert r.$ Let $s\leq p_Y, r,$ then $s\in B_X,$ because $s\force \dot{b}\restr X= f_Y\restr X= f_X= b\restr X.$ 
Since $\bigcup\{X\cap \mu; X\in E\}= \mu,$ we have $q\force \dot{b}= b.$
\end{proof}

That completes the proof. \end{proof}

\section{Cummings and Foreman's Iteration}\label{mainforcing}

In this section we discuss a forcing construction which is due to Cummings and Foreman \cite{CummingsForeman}. We will prove, in \S \ref{theorem}, that this iteration produces a model where every $\aleph_n$ (with $n\geq 2$) satisfies the super tree property. A few considerations will help the reader to understand the definition of this iteration. The standard way to produce a model of the super tree property for $\aleph_{n+2}$ (where $n<\omega$) is the following: we start with a supercompact cardinal $\kappa$ -- by Magidor's theorem it is inaccessible and it satisfies the super tree property --, then we turn $\kappa$ into $\aleph_{n+2}$ by forcing with a poset that preserves the super tree property at $\kappa.$ The forcing notion required for that, is a variation of an iteration due to Mitchell that we denote $\mathbb{M}(\aleph_n, \kappa)$ (see \cite{Mitchell}). A naive attempt to construct a model where the super tree property holds simultaneously for two cardinals $\aleph_{n+2}$ and $\aleph_{n+3},$ would be to start with two supercompact cardinals $\kappa<\lambda,$ and force with $\mathbb{M}(\aleph_n, \kappa)$ first, and then with $M(\aleph_{n+1}, \lambda).$ The problem with that approach is that, at the second step of this iteration, we could lose the super tree property at $\kappa,$ that is at $\aleph_{n+2}.$ For this reason, the first step of the iteration must be reformulated so that, not only it will turn $\kappa$ into $\aleph_{n+2}$ by preserving the super tree property at $\kappa,$ but it will also ``anticipate a fragment" of 
$\mathbb{M}(\aleph_{n+1}, \lambda).$ We are going to define a forcing $\mathbb{R}(\tau, \kappa, V, W, L)$ that will constitutes the main brick of the Cummings and Foreman's iteration. 
If $\kappa$ is supercompact cardinal in the model $V,$ then $\mathbb{R}(\tau, \kappa, V, W, L)$ turns $\kappa$ into $\tau^{++}$ and it makes $\tau^{++}$ satisfy the super tree property in a larger model $W.$ The parameter $L$ refers to the Laver function for $\kappa$ (which is in $V$), such function will be used to ``guess" a fragment of the forcing, at the next step of the iteration, that will be defined in the model $W.$\\
 
None of the results of this section are due to the author.       


\begin{definition}\label{CF}
Let $V\subseteq W$ be two models of set theory and suppose that for some $\tau, \kappa,$ we have 
$W\models (\tau<\kappa \textrm{ is regular and }\kappa \textrm{ is inaccessible}).$ Let $\mathbb{P}:= \Add(\tau, \kappa)^V$ and suppose that 
$W\models \mathbb{P} \textrm{ is $\tau^+$-c.c. and $<\tau$-distributive}.$ Let $L\in W$ be a function with $L: \kappa\to (V_{\kappa})^W.$ Define in $W$ a forcing 
$$\mathbb{R}:= \mathbb{R}(\tau, \kappa, V, W, L)$$
as follows. The definition is by induction; for each $\beta\leq \kappa$ we will define a forcing $\mathbb{R}\restr \beta$ and we will finally set $\mathbb{R}:=\mathbb{R}\restr \kappa.$
$\mathbb{R}\restr 0$ is the trivial forcing.\\
$(p,q,f)$ is a condition in $\mathbb{R}\restr \beta$ if, and only if, 
\begin{enumerate}
\item $p\in \mathbb{P}\restr \beta:= \Add(\tau,\beta)^V;$
\item $q$ is a partial function on $\beta,$ $\vert q\vert\leq \tau,$ $\dom(q)$ consists of successor ordinals, and if $\alpha\in \dom(q),$ then $q(\alpha)\in W^{\mathbb{P}\restr \alpha}$ and $\force_{\mathbb{P}\restr \alpha}^W q(\alpha)\in \Add(\tau^+)^{W[\mathbb{P}\restr \alpha]};$
\item $f$ is a partial function on $\beta,$ $\vert f\vert\leq \tau,$ $\dom(f)$ consists of limit ordinals and $\dom(f)$ is a subset of 
$$\{\alpha;\ \force_{\mathbb{R}\restr \alpha}^W L(\alpha) \textrm{ is a canonically $\tau^+$-directed closed forcing }  \}$$
\item If $\alpha\in \dom(f),$ then $f(\alpha)\in W^{\mathbb{R}\restr \alpha}$ and $\force_{\mathbb{R}\restr \alpha}^W f(\alpha)\in L(\alpha).$
\end{enumerate} 

The conditions in $\mathbb{R}\restr \beta$ are ordered in the following way: 
$$(p', q', f')\leq (p, q, f)$$
if, and only if, 
\begin{enumerate}
\item $p'\leq p;$
\item for all $\alpha\in\dom(q),$ $p'\restr \alpha\force q'(\alpha)\leq q(\alpha);$
\item for all $\alpha\in \dom(f),$ $(p', q', f')\restr \alpha\force_{\mathbb{R}\restr \alpha}^W f'(\alpha)\leq f(\alpha). $
\end{enumerate}
\end{definition}

Here after, some easy property of that forcing. 

\begin{lemma}\label{elle} In the situation of Definition \ref{CF}, $\mathbb{R}$ can be projected to $\mathbb{P},$ $\mathbb{R}\restr \alpha\ast L(\alpha),$ and $\mathbb{P}\restr \alpha\ast \Add(\tau^+)^{W[\mathbb{P}\restr \alpha]}.$
\end{lemma} 

\begin{proof} See \cite[Lemma 3.3]{CummingsForeman}. \end{proof}

The proof of the following lemma is analogous to the proof of \cite[Lemma 3.6]{CummingsForeman}, we include it for completeness. 

\begin{lemma}\label{tantobrava} In the situation of Definition \ref{CF}, if $g\subseteq \mathbb{P}$ is a generic filter and if $\mathbb{P}$ is $<\tau$-distributive in $W,$ then $\mathbb{R}/ g$ is 
$\tau$-directed closed in $W[g].$ In particular, if $\mathbb{P}$ is $\tau$-closed, then $\mathbb{R}$ is $\tau$-closed.  
\end{lemma}

\begin{proof} In $W[g],$ let $\langle (p_i, q_i, f_i);\ i<\gamma \rangle$ be a sequence of less than $\tau$ conditions of $\mathbb{R}.$ Since $\mathbb{P}$ is $<\tau$-distributive, the sequence belongs to $W.$ By definition of $\mathbb{R}/ g,$ we have $p_i\in g$ for every $g,$ so we can fix $p\leq p_i,$ for every $i<\gamma$ (take for example $p\in g$ such that $p\force p_i\in \dot{g} \textrm{ for all }i$). We define a function $q$ with domain $\underset{i<\gamma}{\bigcup}\dom{q_i}$ as follows. For every $\alpha\in \dom(q_{i_{\alpha}}),$ we have $$p\restr \alpha\force \langle q_i(\alpha);\ i_{\alpha}\leq i<\gamma \rangle \textrm{ are pairwise compatible conditions in }\Add(\tau^+)^{W[\mathbb{P}\restr \alpha]}.$$ Therefore, there is 
$q(\alpha)\in W^{\mathbb{P}\restr \alpha}$ such that $p\restr \alpha\force q(\alpha)\leq q_i(\alpha) \textrm{ for every }i<\gamma.$ 
Now we define a function $f$ with domain $\underset{i<\gamma}{\bigcup}\dom(f_i).$ We define $f(\alpha),$ by induction on $\alpha,$ so that 
$(p, q, f)\restr \alpha$ is a lower bound for the sequence $\langle (p_i, q_i, f_i)\restr \alpha;\ i<\gamma \rangle.$ Assume that $f(\beta)$ has been defined for every $\beta<\alpha,$ then $$(p,q,f)\restr \alpha\force \langle f_i(\alpha);\ i<\gamma \rangle \textrm{ are pairwise compatible conditions in }L(\alpha).$$ By definition, $L(\alpha)$ is a name for a $\tau^+$-directed closed forcing in $W[\mathbb{R}\restr \alpha],$ so there is $f(\alpha)\in W^{\mathbb{R}\restr \alpha}$ such that 
$(p,q,f)\restr \alpha\force f(\alpha)\leq f_i(\alpha), \textrm{ for every }i<\gamma.$ That completes the definition of $f.$ Finally, the condition $(p,q,f)$ is a lower bound for the sequence $\langle (p_i, q_i, f_i);\ i<\gamma \rangle.$ \end{proof}

\begin{definition}\label{CFiteration} (Cummings and Foreman's Iteration) We consider $\langle \kappa_n;\ n<\omega \rangle$ an increasing sequence of supercompact cardinals. For every $n<\omega,$ let $L_n: \kappa_n\to V_{\kappa_n}$ be the Laver function for $\kappa_n.$ We define by induction a forcing iteration $\mathbb{R}_{\omega}$ of length $\omega$ and 
we let $G_{\omega}$ be a generic filter for $\mathbb{R}_{\omega}$ over $V.$ 
\begin{enumerate}
\item The first stage of the iteration $\mathbb{R}_1$ is $\mathbb{Q}_0:= \mathbb{R}(\aleph_0, \kappa_0, V, V, L_0);$ we let $G_0\subseteq \mathbb{Q}_0$ be generic over $V;$
\item we define a $\mathbb{Q}_0$-name $\dot{L}_1$ as follows, we let
$\dot{L}_1^{G_0}(\alpha):= L_1(\alpha)^{G_0}$ if $L_1(\alpha)$ is a $\mathbb{Q}_0$-name and $\dot{L}_1^{G_0}(\alpha):=0,$ otherwise.  Then 
$\dot{\mathbb{Q}}_1$ is the canonical name for $\mathbb{R}(\aleph_1^V, \kappa_1, V, V[G_0], \dot{L_1}^{G_0}).$ 
We let $\mathbb{R}_2:= \mathbb{Q}_0\ast \dot{\mathbb{Q}}_1$ and we fix $G_1\subseteq \mathbb{Q}_1$ generic over $V[G_0].$
\item Suppose $\mathbb{R}_n:= \mathbb{Q}_0\ast ...\ast \dot{\mathbb{Q}}_{n-1}$ and $G_0, ...G_{n-1}$ have been defined. We define an $\mathbb{R}_n$-name $\dot{L}_n$ by 
$\dot{L}_n^{G_n}(\alpha):= L_n(\alpha)$ if $L_n(\alpha)$ is a $\mathbb{R}_n$-name and $\dot{L}_n^{G_n}(\alpha):=0,$ otherwise. Then, let $V_{n-1}:= V[G_0]...[G_{n-1}]$ and 
let $\dot{\mathbb{Q}}_n$ be a name for 
$\mathbb{R}(\kappa_{n-2}, \kappa_n, V_{n-2}, V_{n-1}, L^*_n),$ where $L^*_n$ is the interpretation of $\dot{L}_n$ in $V_{n-1}.$
Finally, we let $\mathbb{R}_{n+1}:= \mathbb{Q}_0\ast ...\ast \dot{\mathbb{Q}}_n$ and we fix $G_n\subseteq \mathbb{Q}_{n}$ generic over $V_{n-1}.$
\item $\mathbb{R}_{\omega}$ is the inverse limit of $\langle \mathbb{R}_n;\ n<\omega \rangle.$ 
\end{enumerate}
\end{definition}

The following lemma will prove that the previous definition is legitimate. In the statement of the lemma, when we refer to ''$\aleph_i$'' we mean $\aleph_i$ in the sense of 
$V[\mathbb{R}_n].$

\begin{lemma}\label{fondamentale} Let $n\geq 1,$ in $V[\mathbb{R}_n]$ we define $\mathbb{P}_n:= \Add(\aleph_n, \kappa_n)^{V[\mathbb{R}_{n-1}]}$ and 
$\mathbb{U}_n:= \{(0,q,f);\ (0,q,f)\in \mathbb{Q}_n \},$ ordered as a subset of $\mathbb{Q}_n.$ The following hold: 
\begin{enumerate}
\item $V[\mathbb{R}_n]\models 2^{\aleph_i}=\aleph_{i+2}= \kappa_i, \textrm{ for $i<n,$ and $\kappa_j$ is inaccessible for every $j\geq n.$}$
\item $V[\mathbb{R}_n]\models \mathbb{Q}_n \textrm{ is $<\aleph_n$-distributive, $\kappa_n$-Knaster, $\aleph_{n-1}$-directed closed and has size $\kappa_n.$}$ 
\item All cardinals up to $\aleph_{n+1}$ are preserved in $V[\mathbb{R}_n\ast \dot{\mathbb{Q}}_n].$
\item $V[\mathbb{R}_n]\models (\mathbb{Q}_n \textrm{ is a projection of $\mathbb{P}_n\times \mathbb{U}_n),$}$ and 
$V[\mathbb{R}_n\ast \dot{\mathbb{P}}_n]\subseteq V[\mathbb{R}_n\ast \dot{\mathbb{Q}}_n]\subseteq V[\mathbb{R}_n\ast (\dot{\mathbb{P}}_n\times \dot{\mathbb{U}}_n)].$
\item $V[\mathbb{R}_n]\models \mathbb{P}_n\times \mathbb{U}_n \textrm{ is $\kappa_n$-c.c.}$
\item $V[\mathbb{R}_n]\models \mathbb{U}_n \textrm{ is $\aleph_{n+1}$-directed closed and $\kappa_n$-c.c.}$
\item Let $K_n\ast G_n\subseteq \mathbb{R}_n\ast \dot{\mathbb{Q}}_n$ be any generic filter over $V,$ and in $V[K_n\ast G_n]$ let 
$\mathbb{S}_n:= (\mathbb{P}_n\times \mathbb{U}_n)/ G_n,$ then 
$$V[K_n\ast G_n]\models \mathbb{S}_n \textrm{ is $<\aleph_{n+1}$-distributive, $\aleph_n$-closed and $\kappa_n$-c.c.}.$$
\item $\Add(\aleph_n, \eta)^{V[\mathbb{R}_{n-1}]}$ is $\aleph_{n+1}$-Knaster in $V[\mathbb{R}_n\ast \dot{\mathbb{Q}}_n],$ for any ordinal $\eta.$
\item $V[\mathbb{R}_n\ast \dot{\mathbb{Q}}_n]\models \Add(\aleph_{n+1}, \eta)^{V[\mathbb{R}_n]} \textrm{ is $<\aleph_{n+1}$-distributive and $\kappa_n$-Knaster,}$ for any ordinal $\eta.$
\item All $\aleph_n$-sequences of ordinals from $V[\mathbb{R}_n\ast \dot{\mathbb{Q}}_n]$ are in $V[\mathbb{R}_n\ast \dot{\mathbb{P}}_n].$
\end{enumerate}
\end{lemma}

\begin{proof} See \cite[Lemma 4.5]{CummingsForeman} (for claim $5.$ and $10.$ see \cite[Lemma 3.11]{CummingsForeman}, for claim $6.$ see \cite[Lemma 3.8 and 3.9]{CummingsForeman}, finally, claim $7.$ corresponds to \cite[Lemma 3.20]{CummingsForeman})\end{proof}


In the following sections, we will use the previous lemma repeatedly and without comments. 

\begin{definition}\label{provenzano} In the situation of Definition \ref{CF}, let $\beta<\kappa$ and $X_{\beta}$ be $\mathbb{R}\restr \beta$-generic over $W,$ we define  
$\mathbb{R}^*:= \mathbb{R}/ X_{\beta}$ (i.e. $\mathbb{R}^*:= \{ r\in \mathbb{R};\ r\restr \beta\in X_\beta\}).$ $\mathbb{R}^*$ is ordered as a subposet of $\mathbb{R}.$ 
We also let $\mathbb{U}^*:= \{ (0,q,f);\ (0,q,f)\in \mathbb{R}^* \},$
ordered as a suborder of $\mathbb{R}^*.$
Finally, $\mathbb{P}^*:=\{p\in \mathbb{P};\ (p,0,0)\in \mathbb{R}^* \},$ ordered as a suborder of $\mathbb{P}.$
 \end{definition}

\begin{lemma}\label{stella} In the situation of Definition \ref{provenzano}, the following hold: 
\begin{enumerate}
\item the function $\pi: \mathbb{P}^*\times \mathbb{U}^*\to \mathbb{R}^*$ defined by $\pi(p,(0,q,f))\mapsto (p,q,f)$ is a projection;
\item $\mathbb{U}^*$ is $\tau^+$-closed in $W[X_{\beta}].$
\end{enumerate}
\end{lemma}

\begin{proof} See \cite[Lemma 3.24 and 3.25]{CummingsForeman}. \end{proof}


Cummings and Foreman \cite{CummingsForeman} also proved the following lemma. 

\begin{lemma}\label{intermidium} For every $n<\omega,$ let $X\in V[\mathbb{R}_{\omega}]$ be a $\kappa_n$-sequence of ordinals, then 
$X\in V[\mathbb{R}_{n+2}\ast \dot{\mathbb{P}}_{n+2}]$ 
\end{lemma}

\begin{proof} For every $m<\omega,$ if $K_{m+3}$ is any generic filter for $\mathbb{R}_{m+3}$ over $V,$ then $\mathbb{R}_{\omega}/ K_{m+3}$ is $\kappa_m$-closed. Therefore, $X\in V[\mathbb{R}_{n+4}].$ 
Since $\mathbb{Q}_{n+3}$ is $<\kappa_{n+1}$-distributive in $V[\mathbb{R}_{n+4}],$ we have $X\in V[\mathbb{R}_{n+3}].$ Finally, every $\kappa_n$-sequence of ordinals in 
$V[\mathbb{R}_{n+3}]$ is in $V[\mathbb{R}_{n+2}\ast \dot{\mathbb{P}}_{n+2}],$ that completes the proof. \end{proof}


In \cite{CummingsForeman}, this was used to prove that if $T$ is a $\kappa_n$-tree in $V[\mathbb{R}_{\omega}],$ then $T\in V[\mathbb{R}_{n+2}\ast \dot{\mathbb{P}}_{n+2}].$ 
The same property does not hold for $(\kappa_n,\mu)$-trees unless $\mu^{<\kappa_n}$ is large enough.

\section{Expanding Cummings and Foreman's Model}\label{subsec}

To prove the main theorem, we need to expand Cummings and Foreman's model. Recall that $G_{\omega}$ is a generic filter for $\mathbb{R}_{\omega}$ over $V.$ We defined 
$\mathbb{P}_n:= \Add(\aleph_n, \kappa_n)^{V[\mathbb{R}_{n-1}]}$ and $\mathbb{U}_n:= \{(0,q,f);\ (0,q,f)\in \mathbb{Q}_n \}$ is ordered as a subset of $\mathbb{Q}_n.$ For every $n<\omega,$ if $K_n\ast G_n$ is any generic filter for $\mathbb{R}_n\ast \mathbb{Q}_n$ over $V,$ 
then $\mathbb{S}_n$ is a forcing notion in $V[K_n\ast G_n]$ and it denotes $(\mathbb{P}_n\times \mathbb{U}_n)/ G_n$ (see Lemma \ref{fondamentale}). In this section we observe what happens when we force over $V[G_{\omega}]$ with $\mathbb{S}_{n+1}$ and then with $\mathbb{S}_{n+2}.$

\begin{definition} For every $n<\omega,$ let $K_{n+1}$ be any generic filter for $\mathbb{R}_{n+1}$ over $V.$ We define in $V[K_{n+1}]$ the forcing 
$$\tail_{n+1}:= \mathbb{R}_{\omega}/ K_{n+1}.$$ 
\end{definition}

\begin{remark}\label{diritti}
By Lemma \ref{fondamentale}, $\tail_{n+3}$ is a 
$\kappa_n$-directed closed forcing in $V[\mathbb{R}_{n+3}].$ Since $\mathbb{S}_{n+2}$ is $\kappa_n$-closed,  
the poset $\tail_{n+3}$ is $\kappa_n$-directed closed even in $V[\mathbb{R}_{n+3}\ast \dot{\mathbb{S}}_{n+2}]= V[\mathbb{R}_{n+2}\ast (\dot{\mathbb{P}}_{n+2}\times \dot{\mathbb{U}}_{n+2})].$
\end{remark}

Recall that $G_{\omega}\subseteq \mathbb{R}_{\omega}$ is the generic filter over $V$ fixed in Definition \ref{CFiteration}. 

\begin{notation} From now on, for every $n<\omega,$ we denote by $G_0\ast ... \ast G_n$ the $\mathbb{R}_{n+1}$-generic filter over $V$ derived from $G_{\omega}.$ We also let $V_n:=V[G_0\ast ...\ast G_n].$
\end{notation}


\begin{lemma} $\mathbb{S}_{n+1}$ is $\kappa_{n+1}$-c.c. in $V[\mathbb{R}_{\omega}].$
\end{lemma}

\begin{proof} Suppose, towards a contradiction, that there is a maximal antichain $A\in V[\mathbb{R}_{\omega}]$ for $\mathbb{S}_{n+1}$ which has size $\kappa_{n+1}$ in 
$V[\mathbb{R}_{\omega}].$ By Lemma \ref{intermidium} $A\in [\mathbb{R}_{n+3}\ast \mathbb{P}_{n+3}].$ The poset $\mathbb{P}_{n+3}$ is $\kappa_{n+2}$-c.c. in $V[\mathbb{R}_{n+3}],$ so $A$ is covered by some antichain $A'$ of size $\leq \kappa_{n+1}$ in $V[\mathbb{R}_{n+3}].$ By maximality of $A,$ we have $A= A',$ so $A\in V[\mathbb{R}_{n+3}].$ Now, $\mathbb{Q}_{n+2}$ is $\kappa_{n+2}$-c.c. so, by the same argument, we have $A\in V[\mathbb{R}_{n+1}],$ but $\mathbb{S}_{n+1}$ is $\kappa_{n+1}$-c.c. in that model, that leads to a contradiction. \end{proof}

\begin{lemma} Every maximal antichain of $\mathbb{S}_{n+1}$ in $V[\mathbb{R}_{\omega}],$ belongs to $V[\mathbb{R}_{n+2}].$ 
\end{lemma}

\begin{proof} Let $A$ be a maximal antichain of $\mathbb{S}_{n+1}$ in $V[\mathbb{R}_{\omega}],$ then $A$ has size $\leq \kappa_n.$ By Lemma \ref{intermidium}, then 
$A\in V[\mathbb{R}_{n+2}\ast \dot{\mathbb{P}}_{n+2}].$ Since $\mathbb{P}_{n+2}$ is $\kappa_{n+1}$-c.c., $A$ is covered by an antichain $A'$ of size $\leq \kappa_n$ which is in 
$V[\mathbb{R}_{n+2}];$ by the maximality of $A,$ we have $A= A'.$ \end{proof}

By the previous lemma, when we force with $\mathbb{S}_{n+1}$ over $V[G_{\omega}],$ we can look at the resulting extension as being obtained from $V_{n+1}$ by forcing first with $\mathbb{S}_{n+1}$ and then with $\tail_{n+2}.$ That justifies the following definition. 

\begin{definition}\label{20feb} We denote by $$V_n[g_{n+1}\times u_{n+1}][G_{tail2}]$$
the generic extension obtained by forcing with $\mathbb{S}_{n+1}$ over $V[G_{\omega}],$ where\linebreak 
$g_{n+1}\times u_{n+1}\subseteq \mathbb{P}_{n+1}\times \mathbb{U}_{n+1}$ is generic over $V_n$ and 
$G_{tail2}\subseteq \tail_{n+2}$ is generic over $V_n[g_{n+1}\times u_{n+1}].$  
\end{definition}

\begin{lemma} Every maximal antichain of $\mathbb{S}_{n+2}$ in $V[\mathbb{R}_{\omega}\ast \dot{\mathbb{S}}_{n+1}],$ belongs to 
$V[\mathbb{R}_{n+1}\ast (\dot{\mathbb{P}}_{n+1}\times \dot{\mathbb{U}}_{n+1})\ast \dot{\mathbb{Q}}_{n+2}].$
\end{lemma}

\begin{proof} $\mathbb{S}_{n+2}$ is $\kappa_{n+2}$-c.c. in $V[\mathbb{R}_{\omega}].$ Since $\mathbb{S}_{n+1}$ is $\kappa_{n+2}$-c.c. in $V[\mathbb{R}_{\omega}],$ the poset $\mathbb{S}_{n+2}$ is $\kappa_{n+2}$-c.c even in $V[\mathbb{R}_{\omega}\ast \dot{\mathbb{S}}_{n+1}].$ Let $A$ be a maximal antichain of $\mathbb{S}_{n+2}$ in 
$V[\mathbb{R}_{\omega}\ast \dot{\mathbb{S}}_{n+1}],$ then $A$ has size $\leq \kappa_{n+1}.$ Since $\mathbb{S}_{n+2}$ is 
$\kappa_{n+2}$-c.c in $V[\mathbb{R}_{\omega}],$ we have $A\in V[\mathbb{R}_{\omega}].$ By Lemma \ref{intermidium}, then, $A\in V[\mathbb{R}_{n+3}\ast \dot{\mathbb{P}}_{n+3}].$ Moreover, $\mathbb{P}_{n+3}$ is $\kappa_{n+2}$-c.c., so $A$ belongs to $V[\mathbb{R}_{n+3}].$ In particular $A$ is in $V[\mathbb{R}_{n+1}\ast (\dot{\mathbb{P}}_{n+1}\times \dot{\mathbb{U}}_{n+1})\ast \dot{\mathbb{Q}}_{n+2}].$ \end{proof}

By the previous lemma, if we force with $\mathbb{S}_{n+2}$ over $V_n[g_{n+1}\times u_{n+1}][G_{tail_2}],$ then the resulting extension can be viewed as obtained 
by forcing first with $\mathbb{S}_{n+2}$ over $V_n[g_{n+1}\times u_{n+1}][G_{n+2}]$ and then with $\tail_{n+3}.$ That justifies the following definition. 

\begin{definition}\label{20feb2} We let $$V_n[g_{n+1}\times u_{n+1}][g_{n+2}\times u_{n+2}][G_{tail3}]$$ 
be the generic extension obtained by forcing with $\mathbb{S}_{n+2}$ over 
$V_n[g_{n+1}\times u_{n+1}][G_{tail_2}],$ 
where $g_{n+2}\times u_{n+2}\subseteq \mathbb{P}_{n+2}\times \mathbb{U}_{n+2}$ is generic over $V_n[g_{n+1}\times u_{n+1}]$ and $G_{tail_3}\subseteq \tail_{n+3}$ is generic over 
$V_n[g_{n+1}\times u_{n+1}][g_{n+2}\times u_{n+2}].$ 
\end{definition}

\section{The Term Forcing}\label{the term forcing}

In the previous section we defined the model $V_n[g_{n+1}\times u_{n+1}][g_{n+2}\times u_{n+2}][G_{tail3}],$ which is the result of forcing over $V_{n+2}$ with the iteration 
$\tail_{n+3}\ast \mathbb{S}_{n+1}\ast \mathbb{S}_{n+2}.$ Now, we want to show that this model can be seen as being obtained by forcing over $V_n$ with a cartesian product that satisfies particular properties. In order to define that forcing notion, first we need to introduce the notion of ``term forcing" (that notion is due to Mitchell \cite{Mitchell06}).

\begin{definition}\label{term forcing}
Let $\mathbb{P}$ be a forcing notion and let $\dot{\mathbb{Q}}$ be a $\mathbb{P}$-name for a poset, we let 
$$\mathbb{T}:=\{\dot{q};\ \force_{\mathbb{P}} \dot{q}\in \dot{Q} \}$$ 
$\mathbb{T}$ is ordered as follows: $\dot{q}\leq^* \dot{r}$ if, and only if, $\force_{\mathbb{P}} \dot{q}\leq \dot{r}.$ The poset $(\mathbb{T}, \leq^*),$ so defined, is called the 
\emph{$\mathbb{P}$-term-forcing for $\dot{\mathbb{Q}}$}.
\end{definition}

\begin{lemma}\label{lemmatf} In the situation of Definition \ref{term forcing}, the following hold:
\begin{enumerate}
\item $\mathbb{P}\ast \dot{\mathbb{Q}}$ is a projection of $\mathbb{P}\times \mathbb{T};$
\item if $\force_{\mathbb{P}} \dot{\mathbb{Q}} \textrm{ is $\kappa$-directed closed},$ then $\mathbb{T}$ is $\kappa$-directed closed as well.
\end{enumerate}
\end{lemma}   

\begin{proof}
\item[$(1)$] Let $\pi: \mathbb{P}\times \mathbb{T}\to \mathbb{P}\ast \dot{\mathbb{Q}}$ be the map $(p,\dot{q})\mapsto (p,\dot{q}),$ we prove that $\pi$ is a projection. It is clear that 
$\pi$ respects the ordering relation and $\pi(1_{\mathbb{P}\times \mathbb{T}})= (1_{\mathbb{P}\ast \dot{\mathbb{Q}}}).$ In $\mathbb{P}\ast \dot{\mathbb{Q}},$ let $(p_0, \dot{q}_0 )\leq (p_1, \dot{q}_1),$ then $p_0\leq p_1$ and $p_0\force \dot{q}_0\leq \dot{q}_1.$ Define $\dot{q}$ as a $\mathbb{P}$-name for an element of $\dot{Q}$ such that for every 
$\mathbb{P}$-generic filter $G,$ we have $\dot{q}^G= \dot{q}_0^G$ if $p_0\in G,$ and $\dot{q}^G= \dot{q}_1^G$ otherwise. Then, $[(p_0, \dot{q})]= [(p_0, \dot{q}_0)].$     
\item[$(2)$] Assume that  
$\langle \dot{q}_{\alpha};\ \alpha <\gamma\rangle$ is a sequence of less than $\kappa$ pairwise compatible conditions in $\mathbb{T}.$ Then, 
$$ \force_{\mathbb{P}} ``\langle \dot{q}_{\alpha};\ \alpha<\gamma \rangle \textrm{ are pairwise compatible conditions in $\dot{\mathbb{Q}}",$}$$  hence 
there exists a $\mathbb{P}$-name $\dot{q}$ such that $ \force_{\mathbb{P}} \dot{q}\leq \dot{q}_{\alpha},$ for every $\alpha<\gamma.$ This means that 
$\dot{q}\leq^* \dot{q}_{\alpha},$ 
for every $\alpha<\gamma.$ \end{proof}

\begin{definition}
In $V[\mathbb{R}_{n+2}],$ we define $\mathbb{T}_{n+3}$ as the $(\mathbb{P}_{n+2}\times \mathbb{U}_{n+2})$-term-forcing for $\tail_{n+3}.$ In $V[\mathbb{R}_{n+1}],$ we let 
$\mathbb{T}_{n+2}$ be the $(\mathbb{P}_{n+1}\times \mathbb{U}_{n+1}\times \mathbb{P}_{n+2})$-term-forcing for the poset $\dot{\mathbb{U}}_{n+2}\times \mathbb{T}_{n+3}.$ 
\end{definition}

\begin{lemma} The following hold: 
\begin{enumerate}
\item $\mathbb{T}_{n+3}$ is $\kappa_n$-directed closed in $V[\mathbb{R}_{n+1}\ast (\dot{\mathbb{P}}_{n+1}\times \dot{\mathbb{U}}_{n+1})];$
\item $\mathbb{T}_{n+2}$ is $\kappa_n$-directed closed in $V[\mathbb{R}_{n+1}].$
\end{enumerate}
\end{lemma}

\begin{proof}

\item[($1$)] We already observed (see Remark \ref{diritti}) that $\tail_{n+3}$ is $\kappa_n$-directed closed in $V[\mathbb{R}_{n+2}\ast (\dot{\mathbb{P}}_{n+2}\times \dot{\mathbb{U}}_{n+2})].$ By Lemma \ref{lemmatf}, then,  
$\mathbb{T}_{n+3}$ is $\kappa_n$-directed closed in 
$V[\mathbb{R}_{n+2}].$ Finally, $\mathbb{S}_{n+1}$ is $<\kappa_n$-distributive, so $\mathbb{T}_{n+3}$ is $\kappa_n$-directed closed in 
$V[\mathbb{R}_{n+1}\ast (\mathbb{P}_{n+1}\times \mathbb{U}_{n+1})];$

\item[($2$)] By the previous claim, the product $\dot{\mathbb{U}}_{n+2}\times \mathbb{T}_{n+3}$ is $\kappa_n$-directed closed in $V[\mathbb{R}_{n+2}].$ 
The poset $\mathbb{S}_{n+1}$ is $<\kappa_n$-distributive in $V[\mathbb{R}_{n+2}],$ so
$\dot{\mathbb{U}}_{n+2}\times \mathbb{T}_{n+3}$ is $\kappa_n$-directed closed in $V[\mathbb{R}_{n+1}\ast (\dot{\mathbb{P}}_{n+1}\times \dot{\mathbb{U}}_{n+1})]$ as well. 
Now, $\mathbb{P}_{n+2}$ is $<\kappa_n$-distributive in $V[\mathbb{R}_{n+1}\ast (\dot{\mathbb{P}}_{n+1}\times \dot{\mathbb{U}}_{n+1})],$ 
so $\dot{\mathbb{U}}_{n+2}\times \mathbb{T}_{n+3}$ is $\kappa_n$-directed closed even in 
$V[\mathbb{R}_{n+1}\ast (\dot{\mathbb{P}}_{n+1}\times \dot{\mathbb{U}}_{n+1})\ast \dot{\mathbb{P}}_{n+2}]=
V[\mathbb{R}_{n+1}\ast (\dot{\mathbb{P}}_{n+1}\times \dot{\mathbb{U}}_{n+1}\times \dot{\mathbb{P}}_{n+2})].$ By Lemma \ref{lemmatf}, then, $\mathbb{T}_{n+2}$ is $\kappa_n$-directed closed in 
$V[\mathbb{R}_{n+1}].$
\end{proof}

 \begin{definition}\label{20feb3} 
 We let 
 $$V_n[g_{n+1}\times u_{n+1}][g_{n+2}\times u_{n+2}\times t_{n+3}]$$ 
be the generic extension obtained by forcing with $\mathbb{T}_{n+3}$ over $V_n[g_{n+1}\times u_{n+1}][g_{n+2}\times u_{n+2}][G_{tail3}],$
 where $t_{n+3}\subseteq \mathbb{T}_{n+3}$ is generic over $V_n[g_{n+1}\times u_{n+1}].$ 
We let  
$$V_n[g_{n+1}\times u_{n+1}\times g_{n+2}\times t_{n+2}]$$ 
be the generic extension obtained by forcing with $\mathbb{T}_{n+2}$ over $V_n[g_{n+1}\times u_{n+1}][g_{n+2}\times u_{n+2}\times t_{n+3}],$
where $t_{n+2}\subseteq \mathbb{T}_{n+2}$ is generic over $V_n.$
\end{definition}

\begin{remark}\label{rmk} The poset $(\mathbb{P}_{n+2}\times \mathbb{U}_{n+2}\times \mathbb{T}_{n+2})/ (g_{n+2}\times u_{n+2})\ast G_{tail3}$ is $\kappa_n$-closed; the poset 
$(\mathbb{P}_{n+1}\times \mathbb{U}_{n+1}\times \mathbb{P}_{n+2}\times \mathbb{T}_{n+2})/ (g_{n+1}\times u_{n+1}\times g_{n+2})\ast (u_{n+2}\times t_{n+3})$ is $\aleph_{n+1}$-closed. \end{remark}

Summing up, we have:
\begin{enumerate}
\item  $V[G_{\omega}]\subseteq V_n[g_{n+1}\times u_{n+1}][G_{tail2}],$ the latter model has been obtained by forcing with $\mathbb{S}_{n+1}$ over $V[G_{\omega}];$
\item $V_n[g_{n+1}\times u_{n+1}][G_{tail2}]\subseteq V_n[g_{n+1}\times u_{n+1}][g_{n+2}\times u_{n+2}][G_{tail3}],$ the latter model has been obtained by forcing with 
$\mathbb{S}_{n+2}$ over the former;  
\item $V_n[g_{n+1}\times u_{n+1}][g_{n+2}\times u_{n+2}][G_{tail3}]\subseteq V_n[g_{n+1}\times u_{n+1}][g_{n+2}\times u_{n+2}\times t_{n+3}],$ the latter model has been obtained by forcing over the former with a $\kappa_n$-closed forcing; 
\item $V_n[g_{n+1}\times u_{n+1}][g_{n+2}\times u_{n+2}\times t_{n+3}]\subseteq V_n[g_{n+1}\times u_{n+1}\times g_{n+2}\times t_{n+2}],$ the latter model has been obtained by forcing over the former with an $\aleph_{n+1}$-closed forcing. 
\end{enumerate}


\section{More Preservation Results}\label{morepreserv}

It will be important, in what follows that the forcing that takes us from $G_{\omega}$ to the model $V_n[g_{n+1}\times u_{n+1}\times g_{n+2}\times t_{n+2}]$ defined in the previous section, cannot add cofinal branches to an $(\aleph_{n+2}, \mu)$-tree. 

\begin{lemma}\label{parteliminata} Let $F$ be an $(\aleph_{n+2},\mu)$-tree in $V[G_{\omega}],$ if $b$ is a cofinal branch for $F$ in $V_n[g_{n+1}\times u_{n+1}\times g_{n+2}\times t_{n+2}],$ then $b\in V[G_{\omega}].$ 
\end{lemma}

\begin{proof} Assume towards a contradiction that $b\notin V[G_{\omega}].$ The forcing $\mathbb{S}_{n+1}$ is $\kappa_{n-1}$-closed in $V_{n+1}$ and, since $\tail_{n+2}$ is $\kappa_{n-1}$-closed, $\mathbb{S}_{n+1}$ remains $\kappa_{n-1}$-closed (that is $\aleph_{n+1}$-closed) in $V[G_{\omega}],$ where $\kappa_n=\aleph_{n+2}= 2^{\aleph_n}.$ By the First Preservation Theorem, we have
$$b\notin V_n[g_{n+1}\times u_{n+1}][G_{tail2}].$$

Now, $\mathbb{S}_{n+2}$ is $\kappa_n$-closed in $V_{n+2}$ and, since $\mathbb{S}_{n+1}$ is $<\kappa_n$-distributive and $\tail_{n+3}$ is $\kappa_n$-closed, the poset $\mathbb{S}_{n+2}$ remains $\kappa_n$-closed (that is $\aleph_{n+2}$-closed) in the model $V_n[g_{n+1}\times u_{n+1}][G_{tail2}].$ Another application of the First Preservation Theorem gives 
$$b\notin V_n[g_{n+1}\times u_{n+1}][g_{n+2}\times u_{n+2}][G_{tail3}].$$ 

The passage from $V_n[g_{n+1}\times u_{n+1}][g_{n+2}\times u_{n+2}][G_{tail3}]$ to $V_n[g_{n+1}\times u_{n+1}][g_{n+2}\times u_{n+2}\times t_{n+3}]$
is done by a $\kappa_n$-closed forcing (see Remark \ref{rmk}), hence by the First Preservation Theorem, we get 
$b\notin V_n[g_{n+1}\times u_{n+1}][g_{n+2}\times u_{n+2}\times t_{n+3}]= V_n[g_{n+1}\times u_{n+1}\times g_{n+2}][u_{n+2}\times t_{n+3}].$
The forcing that takes us from $V_n[g_{n+1}\times u_{n+1}\times g_{n+2}][u_{n+2}\times t_{n+3}]$ to $V_n[g_{n+1}\times u_{n+1}\times g_{n+2}\times t_{n+2}]$ is 
$\aleph_{n+1}$-closed (see Remark \ref{rmk}), hence by the First Preservation Theorem, we have 
$$b\notin V_n[g_{n+1}\times u_{n+1}\times g_{n+2}\times t_{n+2}],$$
that leads to a contradiction. \end{proof}

For the proof of the final theorem, we will also need the following lemma.  

\begin{lemma}\label{preservRstella} Let $\mathbb{R}:= \mathbb{R}(\tau, \kappa, V, W, L)$ be like in Definition \ref{CF} and let $\theta<\kappa$ be such that:
\begin{enumerate}
\item for some $n<\omega,$ $\tau= \aleph_n$ and $\aleph_m^V= \aleph_m^W,$ for every $m\leq n.$
\item in $W$ we have $\gamma^{<\tau}<\theta,$ for every $\gamma<\theta,$ 
\end{enumerate}
For every $(\theta, \mu)$-tree $F$ in $W[\mathbb{R}\restr \theta],$
if $b$ is a cofinal branch for $F$ in $W[\mathbb{R}],$ then $b\in W[\mathbb{R}\restr \theta].$
\end{lemma}

\begin{proof} Let $G$ be any $\mathbb{R}$-generic filter over $W.$ Assume towards a contradiction that $b\notin W[G_\theta],$ where 
$G_{\theta}= G\restr \theta.$  
By Lemma \ref{stella}, the forcing $\mathbb{R}^*:= \mathbb{R}/ G_{\theta}$ is a projection of $\mathbb{P}^*\times \mathbb{U}^*,$ where 
$\mathbb{P}^*= \Add(\tau, \kappa-\theta)^{V}$ and $\mathbb{U}^*$ is $\tau^+$-closed in $W[G_{\theta}].$ Let $g^*\times u^*\subseteq \mathbb{P}^*\times \mathbb{U}^*$ be any generic filter over $W$ that projects on $G.$ 
We have $\theta= \tau^{++}= 2^{\tau}$ in $W[G_{\theta}],$ and $F$ is a $(\theta, \mu)$-tree.  Therefore, we can apply the First Preservation Theorem, hence $b\notin W[G_{\theta}][u^*].$ The filter $u^*$ collapses $\theta$ to $\tau^+,$ so now $F$ is a $(\tau^+, \mu)$-tree in $W[G_{\theta}][u^*].$ We want to use the Second Preservation Theorem to prove that $\mathbb{P}^*$ cannot add cofinal branches to $W[G_{\theta}][u^*].$ We can see $\mathbb{P}^*$ as a subset of 
$\Add(\tau, \kappa)^{W[G_{\theta}]}.$ By hypothesis, $W\models \gamma^{<\tau}<\theta,$ for every $\gamma<\theta.$ Moreover, $\mathbb{R}\restr {\theta}$ is $<\tau$-distributive and $\theta$-c.c., so $W[G_{\theta}]\models \gamma^{<\tau}<\theta,$ for every $\gamma<\theta.$ Since $\mathbb{U}^*$ is $\tau^+$-closed, the pair 
$(W[G_{\theta}], W[G_{\theta}][u^*] )$ satisfies condition $(3)$ of the Second Preservation Theorem. So, all the hypothesis of the Second Preservation Theorem are satisfied, hence 
$b\notin W[G_{\theta}][u^*][g^*]$ in particular $b\notin W[G].$ That completes the proof of the lemma. \end{proof}

\section{The Final Theorem}\label{theorem} 

\begin{theorem}\label{teorema} In $V[G_{\omega}],$ every cardinal $\aleph_{n+2}$ has the super tree property.
\end{theorem}

\begin{proof} Let $F\in V[G_{\omega}]$ be an $(\aleph_{n+2}, \mu)$-tree and let $D$ be an $F$-level sequence. In $V[G_{\omega}],$ we have $\kappa_n= \aleph_{n+2},$ 
so $F$ is a $(\kappa_n, \mu)$-tree. We start working in $V.$ Let $\lambda: \sup_{n<\omega} \kappa_n$ and fix $\nu$ grater than both $\mu^{<\kappa_n}$ and 
$\lambda^{\omega}.$ There is an elementary embedding 
$j: V\to M$ with critical point $\kappa_n$ such that:
\begin{enumerate}
\item[(i)] $j(\kappa_n)>\nu$ and ${}^{<\nu}M\subseteq M;$
\item[(ii)] $j(L_n)(\kappa_n)$ is the canonical $\mathbb{R}_n$-name for the canonical $\mathbb{Q}_n$-name for the forcing 
$$\mathbb{U}_{n+1}\times \mathbb{P}_{n+2}\times \mathbb{T}_{n+2}.$$

\end{enumerate}

 
 Note that $j(L)(\kappa_n)$ is a name for a $\kappa_n$-directed closed forcing in 
 $V[\mathbb{R}_n\ast \dot{\mathbb{Q}}_n].$ 

 The proof of the theorem consists of three parts: 
\begin{enumerate}
\item we show that we can lift $j$ to get an elementary embedding 
$$j^*: V[G_{\omega}]\to M[H_{\omega}],$$ where $H_{\omega}\subseteq j(\mathbb{R}_{\omega})$ is generic over $V;$
\item we prove that there is, in $M[H_{\omega}],$ an ineffable branch $b$ for $D;$ 
\item we show that $b\in V[G_{\omega}].$  
\end{enumerate}

\section*{Part 1}

We prove Claim $1.$ To simplify the notation we will denote all the extensions of $j$ by ``$j$" also. Recall that  
$$V[G_{\omega}]\subseteq V_n[g_{n+1}\times u_{n+1}][G_{tail2}]\subseteq V_n[g_{n+1}\times u_{n+1}][g_{n+2}\times u_{n+2}][G_{tail3}]$$ 
$$\subseteq V_n[g_{n+1}\times u_{n+1}][g_{n+2}\times u_{n+2}\times t_{n+3}]\subseteq V_n[g_{n+1}\times u_{n+1}\times g_{n+2}\times t_{n+2}].$$

The forcing $\mathbb{R}_n$ has size less than $\kappa_n,$ so we can lift $j$ to get an elementary embedding 
$$j: V_{n-1}\to V_{n-1}.$$ 

For every $i<\omega,$ we denote by $M_i$ the model $M[G_0]...[G_i].$ We will use repeatedly and without comments the resemblance between $V$ and $M.$ In $M_{n-1},$ we have $$j(\mathbb{Q}_n)\restr \kappa_n= \mathbb{Q}_n,$$
and at stage $\kappa_n,$ the forcing at the third coordinate will be $j(L_n)(\kappa_n)$ (see Lemma \ref{elle}).  By our choice of $j(L_n)(\kappa_n),$ this means that we can look at the model 
$M_n[u_{n+1}\times g_{n+2}\times t_{n+2}]$ as a generic extension of $M_{n-1}$ by 
$j(\mathbb{Q}_{n})\restr \kappa_n+1.$ Force with $j(\mathbb{Q}_n)$ over $W$ to get a generic filter $H_n$ such that 
$H_n\restr \kappa_n+1= G_n\ast (u_{n+2}\times g_{n+2}\times t_{n+2}).$ The forcing $\mathbb{Q}_n$ is $\kappa_n$-c.c. in 
$M_{n-1},$ so $j\restr \mathbb{Q}_n$ is a complete embedding from $\mathbb{Q}_n$ into $j(\mathbb{Q}_n).$ Consequently, we can lift $j$ to get an elementary embedding 
$$j: V_n\to M_{n-1}[H_n].$$  

We know that $\mathbb{P}_{n+1}$ is $\kappa_n$-c.c. in $V_n,$ hence $j\restr \mathbb{P}_{n+1}$ is a complete embedding from $\mathbb{P}_{n+1}$ into 
$j(\mathbb{P}_{n+1}):= \Add(\aleph_{n+1}, j(\kappa_{n+1}))^{V_{n-1}}.$ $\mathbb{P}_{n+1}$ is even isomorphic via $j$ to 
$\Add(\aleph_{n+1}, j[\kappa_{n+1}])^{V_{n-1}}.$ Force with $Add(\aleph_{n+1}, j(\kappa_{n+1})- j[\kappa_{n+1}])^{V_{n-1}}$ over $V_n[H_n][g_{n+1}]$ to get a generic filter $h_{n+1}\subseteq j(\mathbb{P}_{n+1})$ such that $j[g_{n+1}]\subseteq h_{n+1}.$ We can lift $j$ to get an elementary embedding 
$$j: V_n[g_{n+1}]\to M_{n-1}[H_n][h_{n+1}].$$

By the previous observations and the closure of $M,$ we have $j[u_{n+1}\times g_{n+2}\times t_{n+2}]\in M_{n-1}[H_n].$ 
The filter $H_n$ collapses every cardinal below $j(\kappa_n)$ to have size $\aleph_{n+1}$ in $M_{n-1}[H_n],$ therefore the set $j[u_{n+1}\times g_{n+2}\times t_{n+2}]$ has size 
$\aleph_1$ in that model. Moreover, $j(\mathbb{U}_{n+1})\times j(\mathbb{P}_{n+2})\times j(\mathbb{T}_{n+2})$ is a 
$j(\kappa_n)$-directed closed forcing and $j(\kappa_n)= \aleph_{n+2}^{M_{n-1}[H_n]}.$ So, we can find a condition $t^*$ stronger than every condition 
$j(q)\in j[u_{n+1}\times g_{n+2}\times t_{n+2}].$ By forcing over 
$V_{n-1}[H_{n}][h_n]$ with $j(\mathbb{U}_{n+1})\times j(\mathbb{P}_{n+2})\times j(\mathbb{T}_{n+2})$ below $t^*$ we get a generic filter $ x_{n+1}\times h_{n+2}\times \l_{n+2}$ such that 
$j[u_{n+1}]\subseteq x_{n+1},$ $j[g_{n+2}]\subseteq h_{n+2}$ and $j[t_{n+2}]\subseteq \l_{n+2}.$ By Easton's Lemma the filters 
$h_{n+1}$ and $x_{n+1}\times  h_{n+2}\times \l_{n+2}$ are mutually generic over $M_{n-1}[H_n],$ and 
$h_{n+1}\times x_{n+1}$ generates a filter $H_{n+1}$ generic for $j(\mathbb{Q}_{n+1})$ over 
$M_{n-1}[H_n].$ By the properties of projections, we have $j[G_{n+1}]\subseteq H_{n+1}.$ Therefore, the embedding $j$ lifts to an elementary embedding 
$$j: V_{n+1}\to M_{n-1}[H_n][H_{n+1}].$$ 

By definition of $j(\mathbb{T}_{n+2}),$ the filter $h_{n+2}\times l_{n+2}$ determines a generic filter $(h_{n+2}\times x_{n+2})\ast H_{tail3}$ for $(j(\mathbb{P}_{n+2})\times j(\mathbb{U}_{n+2}))\ast j(\tail_{n+3}).$ On the other hand $h_{n+2}\times x_{n+2}$ determines a filter $H_{n+2}$ generic for $j(\mathbb{Q}_{n+2})$ over 
$M_{n-1}[H_n][H_{n+1}].$ By the properties of projections, we have $j[G_{n+2}]\subseteq H_{n+2}.$ Therefore, $j$ lifts to an elementary embedding 
$$j: V_{n+2}\to M_{n-1}[H_{n}][H_{n+1}][H_{n+2}].$$ 

It remains to prove that $j[G_{tail3}]\subseteq H_{tail3},$ but this is an immediate consequence of $j[t_{n+2}]\subseteq \l_{n+2}.$ Finally, $j$ lifts to an elementary embedding 
$$j: V[G_{\omega}]\to M_{n-1}[H_n][H_{n+1}][H_{n+2}][H_{tail3}].$$ 
This completes the proof of Claim $1.$

\section*{Part 2} 

Let $\mathscr{M}_1:= M[G_{\omega}]$ and $\mathscr{M}_2:= M_{n-1}[H_n][H_{n+1}][H_{n+2}][H_{tail3}].$ In $\mathscr{M}_2,$ $j(F)$ is a $(j(\kappa_n), j(\mu))$-tree and 
$j(D)$ is a $j(F)$-level sequence. By the closure of $M,$ the tree $F$ and the $F$-level sequence $D$ are in $\mathscr{M}_1,$ and there is no ineffable branch for $D$ in 
$\mathscr{M}_1.$ We want to find in $\mathscr{M}_2$ an ineffable branch for $D.$ Let $a:= j[\mu],$ clearly $a\in [j(\mu)]^{<j(\kappa_n)}.$ Consider $f:= j(D)(a)$ and let $b: \mu\to 2$ be the function defined by $b(\alpha):= f(j(\alpha)).$ We show that $b$ is an ineffable branch for $D.$ Assume towards a contradiction that for some club 
$C\subseteq [\mu]^{<\vert \kappa_n\vert}\cap \mathscr{N}_1$ we have $b\restr X\neq D(X),$ for all $X\in C.$ By elementarity, $$j(b)\restr X\neq j(D)(X),$$ for all $X\in j(C).$
Observe that $a\in j(C)$ and $j(b)\restr a= f= j(D)(a),$ that leads to a contradiction. 

\section*{Part 3}

 We proved that an ineffable branch $b$ for $D$ exists in $\mathscr{M}_2.$ Now we show that $b\in \mathscr{M}_1,$ thereby proving that $\mathscr{M}_1$ 
 (hence $V[G_{\omega}]$) has an ineffable \footnote{ If $b\in \mathscr{M}_1,$ then $b$ is ineffable since $\{X\in [\mu]^{<\vert \kappa_n\vert }\cap \mathscr{M}_1;\ b\restr X= D(X) \}$ is stationary in $\mathscr{M}_2,$ hence it is stationary in $\mathscr{M}_1.$ }
 branches for $D.$ We will use repeatedly and without comments the resemblance between $V$ and $M.$ 
Assume, towards a contradiction, that $b\notin \mathscr{M}_1.$ Step by step, we are going to prove that $b\notin \mathscr{M}_2.$ By Lemma \ref{parteliminata}, we have $b\notin M_n[g_{n+1}\times u_{n+1}\times g_{n+2}\times t_{n+2}].$ Consider $\Add(\aleph_{n+1}, j(\kappa_{n+1})-j[\kappa_{n+1}])^{M_{n-1}},$ by forcing with this poset over $M_n[g_{n+1}\times u_{n+1}\times g_{n+2}\times t_{n+2}]$ we get
the generic extension $M_n[h_{n+1}\times u_{n+1}\times g_{n+2}\times t_{n+2}];$ we want to prove that $b$ does not belong to that model. 
The pair $(M_{n-1}, M_n[g_{n+1}\times u_{n+1}\times g_{n+2}\times t_{n+2}])$ has the $\kappa_n$-covering property. Moreover, in $V_{n-1},$ the cardinal 
$\kappa_{n}$ is inaccessible, therefore the hypothesis of the Second Preservation Theorem are satisfied and we have 
$$b\notin M_n[h_{n+1}\times u_{n+1}\times g_{n+2}\times t_{n+2}].$$ 

As we said in Part $1,$ we have 
$j(\mathbb{Q}_n)\restr \kappa_{n}= \mathbb{Q}_n,$ and at stage $\kappa_n,$ the forcing at the third coordinate is  
$\mathbb{U}_{n+1}\times \mathbb{P}_{n+2}\times \mathbb{T}_{n+2}.$ 
It follows that for $H^*= H_n\restr \kappa_n+1$ we have just proved 
$$b\notin M_{n-1}[H^*][h_{n+1}]=M_{n-1}[h_{n+1}][H^*].$$

Now we want to show that $\mathbb{R}^*:= j(\mathbb{Q}_n)/ H^*$ cannot add cofinal branches to $F,$ hence $b$ does not belong to the model 
$M_{n-1}[h_{n+1}][H_n].$ It is enough to prove that the hypothesis of Lemma \ref{preservRstella} are satisfied. The cardinal $\kappa_n$ was
inaccessible in $M_{n-1},$ and $h_{n+1}$ is a generic filter for an $\aleph_{n+1}$-closed forcing, so 
$M_{n-1}[h_{n+1}]\models \gamma^{<\aleph_{n+1}}<\kappa_n,$ for every $\gamma<\kappa_n.$ 
Then, it seems that all the hypothesis of Lemma \ref{preservRstella} are satisfied except for the fact that $F$ is not exactly a 
$(\kappa_n,\mu)$-tree in $M_{n-1}[h_{n+1}][H^*]$ because the filter $h_{n+1}$ may add sets in $[\mu]^{<\kappa_n}.$ 
However, the poset $j(\mathbb{P}_{n+1})$ is 
$\kappa_n$-c.c. in $M_{n-1}[H^*],$ so we can say that $F$ covers a $(\kappa_n, \mu)$-tree $F^*$ in $M_{n-1}[H^*][h_{n+1}].$ If $b\in M_{n-1}[H_n][h_{n+1}],$ then $b$ is a cofinal branch for $F^*.$ Then, by Lemma \ref{preservRstella}, we have 
$$b\notin M_{n-1}[h_{n+1}][H_n]= M_{n-1}[H_n][h_{n+1}].$$  

$F^*$ is no longer an $(\aleph_{n+1},\mu)$-tree in $M_{n-1}[H_n][h_{n+1}].$ However, we obtained this model forcing with $j(\mathbb{Q}_n)/ H^*$ which is $\aleph_{n+1}$-c.c. in $M_{n-1}[h_{n+1}][H^*],$ this means that $F^*$ covers an $(\aleph_{n+1}, \mu)$-tree that we can rename $F.$   
Consider $j(\mathbb{Q}_{n+1})/ h_{n+1},$ by Lemma \ref{tantobrava}, this is an $\aleph_{n+1}$-closed forcing in $M_{n-1}[H_n][h_{n+1}],$ where $2^{\aleph_n}\geq j(\kappa_n)= \aleph_{n+2}.$ By the First Preservation Theorem, we have 
$$b\notin M_{n-1}[H_n][H_{n+1}].$$

We continue our analysis by working with $j(\mathbb{Q}_{n+2})$ which is a projection of $j(\mathbb{P}_{n+2})\times j(\mathbb{U}_{n+2}).$ This poset is $\aleph_{n+2}$-closed in 
$M_{n-1}[H_n][H_{n+1}]$ and $F$ is an $(\aleph_{n+1},\mu)$-tree. By the First Preservation Theorem, we have $b\notin M_{n-1}[H_n][H_{n+1}][h_{n+1}\times u_{n+1}],$ in particular
$$b\notin M_{n-1}[H_n][H_{n+1}][H_{n+2}].$$ 

Finally, $j(\tail_{n+3})$ is $\aleph_{n+2}$-closed in $M_{n-1}[H_n][H_{n+1}],$ where $F$ is still an $(\aleph_{n+1}, \mu)$-tree. By applying again the First Preservation Theorem, we get that 
$$b\notin M_{n-1}[H_n][H_{n+1}][H_{n+2}][H_{tail3}]= \mathscr{M}_2,$$

that leads to a contradiction and completes the proof of the theorem. \end{proof}

\bigskip

\end{document}